\documentclass[11pt]{article}
\usepackage[cmex10]{amsmath} 
\usepackage{amssymb}
\usepackage{amsfonts}
\usepackage{amsthm}
\usepackage{setspace}
\usepackage{subfig}
\usepackage[square,numbers]{natbib}
\usepackage[final,colorlinks]{hyperref}
\usepackage[margin=1in]{geometry}

\DeclareMathOperator{\E}{\mathbb{E}}

\newcommand{\email}[1]{\texttt{\href{mailto:#1}{#1}}}
\newcommand{\transpose}{^{\top}}

\newcommand{\Reals}{\mathbb{R}}

\newcommand{\eqD}{\stackrel{\mathcal{D}}{=}} 

\title{A Short Proof of Strassen's Theorem Using Convex Analysis}
\date{\today}
\author{Benjamin Armbruster\thanks{\email{armbrusterb@gmail.com}}
}

\begin{document}
\maketitle

\begin{abstract}
We give a simple proof of Strassen's theorem
on stochastic dominance using linear programming duality, 
without requiring measure-theoretic arguments.
The result extends to generalized inequalities
using conic optimization duality
and provides an additional, intuitive optimization formulation for stochastic dominance.
\end{abstract}



\section{Introduction}
Strassen's theorem (\citeyear{Strassen1965}) is a fundamental theorem in the theory of stochastic dominance.
It characterizes a stochastic dominance relationship
between two random variables 
as an almost sure comparison of the two on the same probability space.
Formally, the theorem states:

\newtheorem*{ST}{Strassen's Theorem}
\begin{ST}
$Z \succeq_{icv} Y$ ($Z \succeq_{cv} Y$) 
if and only if
there exists random variables $Y' \eqD Y$
and $Z' \eqD Z$
such that $Z'\geq \E[Y'|Z']$ 
($Z' = \E[Y'|Z']$) a.s.
\end{ST}

Here $\succeq_{icv}$ denotes dominance in 
increasing concave stochastic order,
also known as second order stochastic dominance;
$\succeq_{cv}$ denotes dominance in concave stochastic order;
and $\eqD$ denotes equality in distribution.
We say that random variable $Z$ dominates random variable $Y$
in increasing concave stochastic order, $Z \succeq_{icv} Y$,
if $\E[u(Z)]\geq \E[u(Y)]$
for all (weakly) increasing concave functions utility functions $u$.
Similarly, we say $Z \succeq_{cv} Y$
if $\E[u(Z)]\geq \E[u(Y)]$
for all concave functions $u$.

Stochastic dominance constraints have been used in optimization problems under uncertainty since \citet{Dentcheva2003}
and have grown in popularity since then (see the +100 papers citing \cite{Dentcheva2003}).
They are often used to ensure that the chosen decision is preferred to some benchmark action.
According to \citet{Lizyayev2012}, there are three categories of constraint formulations: 
distribution based, majorization, and revealed-preference type.  
Strassen's theorem is the basis of the majorization approach (e.g., \eqref{primal}).
The dual formulation in our proof (\eqref{dualprog},\eqref{dualprog2})
is new.
It is similar to the dual program in \cite{Kuosmanen2007} 
but easier to interpret as a comparison of utilities:
let $u$ be a utility function; $a_i=u(y_i)$; $b_j=u(z_j)$; $c_j=u'(z_j)$;
and the constraints ensure that the utility function is concave and increasing.

One benefit of the majorization approach based on Strassen's theorem is that 
it easily extends to vector-valued random variables.
We will consider random variables on $\Reals^k$, with the customary Borel $\sigma$-algebra.
For technical reasons we will assume that the random variables have bounded support.
While this extension is not novel \citep{Sherman1951,Blackwell1953}
and has been used in optimization problems \citep{Armbruster2010dom,haskell:anor13},
we can further extend it to use generalized inequalities, which is novel.
For this, consider a proper (i.e., convex, not containing a line, and closed) cone $\mathcal{K}$.
Then for arbitrary vectors $y$ and $z$ we write the generalized inequality $z\geq_\mathcal{K} y$ if $z-y\in \mathcal{K}$.
This cone also redefines the notion of an increasing function:
$u:\Reals^k\to\Reals$ is an increasing function if $u(z)\geq u(y)$ for all $z\geq_\mathcal{K} y$. 
Increasing functions are of course used to define $\succeq_{icv}$.
Using the nonnegative orthant cone, gives us the usual componentwise inequality and associated definition of an increasing function.
Thus, the precise claim we are proving is as follows.

\newtheorem*{GST}{Generalized Strassen's Theorem}
\begin{GST}
For random vectors $Y$ and $Z$ with bounded support on $\Reals^k$,
$Z \succeq_{icv} Y$ ($Z \succeq_{cv} Y$) 
if and only if
there exists random variables $Y' \eqD Y$
and $Z' \eqD Z$
such that $Z'\geq_\mathcal{K} \E[Y'|Z']$
($Z' = \E[Y'|Z']$) a.s.
\end{GST}

One of the frequent criticisms of stochastic dominance constraints and a motivation for the above generalization
is that the constraints are too conservative 
since they require that $\E[u(Z)]\geq\E[u(Y)]$ for a large class of utility functions.
Here the choice of cone determines the directions in which the utility functions must be increasing,
and this allows us to vary the level of ``conservatism'' in the stochastic dominance constraint.
On one extreme is the cone $\mathcal{K}=\{x : w\cdot x\geq 0\}$, a half-space, 
requiring that the utility function be increasing in all directions making an acute angle to $w$.
In that case $Z\succeq Y$ iff $w\cdot Z \succeq w\cdot Y$ in the sense of scalar stochastic dominance.
Thus, $w$ are the weights at which we trade-off the different components of the outcomes $Y$ and $Z$.
At the other extreme is a very pointed cone $\mathcal{K}=\{w\alpha : \alpha\geq 0\}$ 
only requiring that the utility functions increase in the direction $w$,
and in between is the usual nonnegative orthant cone, $\mathcal{K}=\{x:x\geq 0\}$,
requiring the utility to increase along all coordinate directions.
We can also control the size of the cone (i.e., the level of conservatism) manually, by constructing the convex cone of our choice.

This extension leads to a new connection with conic optimization:
$Z\succeq_{icv} Y$ is equivalent to a feasibility problem in convex optimization problem
involving the cone defining the generalized inequality among vectors (see \eqref{primal}).
For example, this optimization problem is a linear program 
if the generalized inequality is defined by a polyhedral cone,
and it is a semidefinite program,
if we consider the vector space of symmetric matrices 
and the generalized inequality is defined by 
the cone of positive semidefinite matrices.

There is a long history of proofs of Strassen's theorem and extensions \citep{HLP1929,Sherman1951,Blackwell1953,Mirsky1961,Strassen1965,FischerHolbrook1980,EltonHill1992,EltonHill1998,Lindvall1999,Muller2001}.
The theorem was first proved by Hardy, Littlewood, and P\'olya (\citeyear{HLP1929})
for scalar random variables with finite support.
This result was extended to finite-dimensional random vectors by \citet{Sherman1951}.
Most of the proofs appear in the probability literature; use detailed measure-theoretic arguments;
and only consider the case of $\succeq_{cv}$
(which lacks economic intuition because it does not require the utility functions to be increasing).
In contrast, this proof is aimed at researchers in optimization and avoids measure-theoretic arguments.
It is purely geometric, relying essentially on a single use of Farkas' lemma
and is significantly shorter than any of the others.

\section{Proof}\label{sec:proof}
We initially focus on the case of $\succeq_{icv}$
and where $Y$ and $Z$ are scalar and have finite support
on $\{y_i\}$ and $\{z_j\}$, respectively.
In that case, 
\begin{equation}
 \exists\ Y',Z'\ s.t.\ Y' \eqD Y,\ Z' \eqD Y,\ \E[Y'|Z']-Z'\leq 0
\end{equation}
is equivalent term by term to 
\begin{subequations}\label{primal}
\begin{gather}
	\exists\ p_{ij}\ s.t.\quad p_{ij}\geq 0\ \forall i,j,\\
	-\sum_j p_{ij} = -\Pr[Y=y_i]\ \forall i,\quad
	\sum_i p_{ij} = \Pr[Z=z_j]\ \forall j,\quad
	\sum_i p_{ij} (y_i-z_j) \leq 0\ \forall j.\label{primal:ineq}
\end{gather}
\end{subequations}
The equivalence follows from choosing $p_{ij}=\Pr[Y'=y_i,Z'=z_j]$
and applying Bayes' theorem to the last term.
Using slack variables for the inequality in \eqref{primal:ineq}
and then applying Farkas' lemma\footnote{The variant used here is, $Ax=b,\ x\geq 0$ is feasible iff $A\transpose y\geq 0,\ b\transpose y<0$ is infeasible.}, this is equivalent to the infeasible dual system
\begin{subequations}\label{dualprog}
\begin{gather}
	\neg \exists\ a_i,b_j,c_j\ s.t. \\
	a_i \leq b_j + c_j (y_i-z_j)\ \forall i,j, \quad c_j\geq 0\ \forall j,\label{dual:a}\\
	\sum_j b_j \Pr[Z=z_j] - \sum_i a_i \Pr[Y=y_i] < 0.\label{dual:b}
\end{gather}
\end{subequations}
Thus, $Z\succeq_{icv} Y$ can also be written as a linear optimization problem
\begin{subequations}\label{dualprog2}
\begin{align}
	0\leq \min_{a_i,b_j,c_j}\ & \sum_j b_j \Pr[Z=z_j] - \sum_i a_i \Pr[Y=y_i]\\
	s.t.\quad & a_i \leq b_j + c_j (y_i-z_j)\ \forall i,j\\
	& c_j\geq 0\ \forall j.
\end{align}
\end{subequations}
\newtheorem*{lem}{Lemma}
\begin{lem}\label{lem}
The inequalities \eqref{dual:a}--\eqref{dual:b} have a solution iff
there exists functions $u$ and $s$ such that
\begin{subequations}\label{lemeq}
\begin{gather}
	u(x') \leq u(x)+s(x)(x'-x)\ \forall x,x',
	\quad s(x)\geq 0\ \forall x,\label{lemeq:a}\\
	\sum_{x \in \{z_j\}} u(x)\Pr[Z=x] - \sum_{x \in \{y_i\}} u(x)\Pr[Y=x]<0.\label{lemeq:b}
\end{gather}
\end{subequations}
\end{lem}
\begin{proof} 
The ``if'' direction holds by selecting $a_i=u(y_i)$, $b_j=u(z_j)$, and $c_j=s(z_j)$.
Now for the ``only if'' direction, consider a solution of \eqref{dual:a}--\eqref{dual:b}.
Define $u(x)=\min_j\ \{b_j+c_j(x-z_j)\}$.
This function is concave since it is a minimum of affine functions
and increasing since $c_j\geq 0$.
Let $s(x)$ be a supergradient of $u(x)$.
Thus \eqref{lemeq:a} holds.
Since \eqref{dual:a} implies $a_i\leq u(y_i)$
and the definition of $u$ implies $u(z_j)\leq b_j$,
it follows that \eqref{dual:b} implies \eqref{lemeq:b}, proving the claim.
\end{proof}

Note that  \eqref{lemeq:a} 
states that $u$ is an increasing concave utility function 
with supergradient $s$,
and \eqref{lemeq:b} states that $\E[u(Z)]<\E[u(Y)]$.
Thus, \eqref{dualprog} is equivalent to there not existing an 
increasing concave utility function $u$, where $\E[u(Z)]<\E[u(Y)]$.
This is equivalent to 
for all increasing concave utility functions $u$, $\E[u(Z)]\geq\E[u(Y)]$,
which is the definition of $Z\succeq_{icv} Y$ and proves the theorem.

The key to the theorem is duality: 
1) the variables, $p_{ij}$, of the joint probability distribution
are dual to the concavity constraints on the utility function \eqref{dual:a};
and 2) the slopes of the utility function, $c_j$,
are dual to the inequality constraints $Z'\geq \E[Y'|Z']$ in \eqref{primal:ineq}.

The same proof holds for $\succeq_{cv}$
except that now the $c_j$ variables are free.

Essentially the same proof holds for the vector-valued case,
except that we treat the $c_j$ as vectors; 
the product in \eqref{dual:a} as an inner product;
and the constraint $c_j\geq 0$ is meant as $c_j \in \mathcal{K}^*$, 
where $\mathcal{K}^*=\{z:z\cdot y\geq 0\ \forall y\in\mathcal{K}\}$
is the dual cone of $\mathcal{K}$.

We now show the extension to general distributions by more carefully stepping through the application of Farkas' lemma (i.e., going from \eqref{primal} to \eqref{dualprog}) 
using the Hahn-Banach separation theorem.
Let $q_Y$ and $q_Z$ in $L^1(\Reals^k)$ be the distributions of $Y$ and $Z$
and let $S_Y$ and $S_Z$ be their supports.
Thus, $p(y,z)$, the joint distribution of $(Y',Z')$ is in $L^1(\Reals^k\times \Reals^k)$.
We further define slack variables $r(z)\in (L^1(\Reals^k))^k$.
Now we can define the linear operator $A$
\begin{equation}\label{def:A}
A p = \left(-\int_{S_Z} p(y,z) dz,\int_{S_Y} p(y,z) dy,\int_{S_Y} p(y,z)(y-z)1(z \in S_Z) dy \right),\\
\end{equation}
where $1(\cdot)$ denotes the indicator function.
The operator $A$ is bounded because $y-z$ only ranges on some bounded set.
Hence $A$ is a continuous operator.
Hence the following convex cone,
\begin{equation}
C=\{ A p + (0,0,r) : p(y,z)\geq 0\ \forall (y,z), r(z)\in \mathcal{K}\ \forall z, p\in L^1(\Reals^k \times \Reals^k),  r\in (L^1(\Reals^k))^k\},
\end{equation}
is closed.
We define the righthand side of \eqref{primal} as $v=(-q_Y,q_Z,0)$.
Then the separating Hahn-Banach theorem states that exactly one of the following is true: 
$v\in C$ (i.e., \eqref{primal}) or 
there exists a supporting hyperplane (i.e., a continuous linear function) $h$ 
that strictly separates the point $v$ from $C$, that is $h(v)<\alpha\leq h(x)$ for all $x\in C$.
Since $C$ is a cone, we may assume $\alpha=0$.
Since $L^\infty$ is the dual space of $L^1$,
we can define this linear function using a vector $(a,b,c)$ so that $h(x)=(a,b,c)\cdot x$ 
where $a\in L^\infty(\Reals^k)$, $b\in L^\infty(\Reals^k)$ and $c\in (L^\infty(\Reals^k))^k$.
Then $h(v)=(a,b,c)\cdot (-q_Y,q_Z,0)<0$ is equivalent to
\begin{equation}\label{dual:b2}
-\int a(y) q_Y(y) dy + \int b(z) q_Z(z) dz < 0,
\end{equation}
essentially \eqref{dual:b}.
Now $h(x)\geq 0$ for all $x\in C$ iff $(a,b,c)\cdot (Ap)\geq 0$ for all $\{p:p(y,z)\geq 0\ \forall y,z\}$ 
and $c(z)\cdot r(z)\geq 0$ for all $z$ and $r \in \{ r:r(z)\in\mathcal{K}\ \forall z\}$.
Thus $h(x)\geq 0$ for all $x\in C$ iff $(A^* (a,b,c))(y,z) \geq 0$  for all $(y,z)$ 
and $c(z)\in \mathcal{K}^*$ for all $z$.
Here $A^*$ is the adjoint (i.e., transpose) of operator $A$.
Putting the pieces together, $v \in C$ is equivalent to there not existing $(a,b,c)$ such that
\eqref{dual:b2} and
\begin{equation}
-a(y) + b(z) + c(z)\cdot (y-z) \geq 0\ \forall y\in S_Y,z\in S_Z.
\end{equation}
The restrictions to the support come from the indicator functions and the domains of integration in \eqref{def:A}.
This is essentially \eqref{dualprog}.
The rest of the proof follows as before.

\paragraph{Acknowledgments}
I thank Thibaut Barthelemy, Jim Luedtke, and Fatemeh Hamidi Sepehr for suggestions improving the manuscript.

\bibliography{refs}
\end{document}